\documentclass[letterpaper]{article}
\usepackage[utf8]{inputenc}
\usepackage[T1]{fontenc}
\usepackage{hyperref}

\usepackage[auth-sc]{authblk}

\usepackage{verbatim}
\usepackage{graphicx}
\usepackage{rotating}

\usepackage{url}
\usepackage{hyperref}

\usepackage{pslatex}

\usepackage{datetime}

\usepackage{ebproof}

\usepackage{pgf,tikz,environ}
\usetikzlibrary{arrows}
\usetikzlibrary{cd}
\usetikzlibrary{decorations}

\usepackage{amsthm}
\usepackage[fleqn]{amsmath} 
\usepackage{amsfonts}
\usepackage{amssymb}
\usepackage{stix}

\theoremstyle{plain}
\newtheorem{theorem}{Theorem}[section]

\newtheorem{corollary}[theorem]{Corollary}

\newtheorem{proposition}[theorem]{Proposition}

\theoremstyle{definition}
\newtheorem{definition}[theorem]{Definition}
\newtheorem{example}[theorem]{Example}

\newcommand{\SetComp}[2]{\left\{ {#1}\:\middle|\:{#2} \right\}}

\newcommand{\typesetoperator}[1]{\mathbf{#1}}
\newcommand{\typesetCategory}[1]{\mathbf{#1}}
\DeclareMathOperator{\mono}{\typesetoperator{Mono}}
\DeclareMathOperator{\homset}{\typesetoperator{Hom}}

\DeclareMathOperator{\tw}{\typesetoperator{tw}}
\DeclareMathOperator{\id}{\typesetoperator{id}}

\DeclareMathOperator{\GRPH}{\typesetCategory{Gr}}
\DeclareMathOperator{\Rmono}{\typesetCategory{\GRPH}_{R-mono}}
\DeclareMathOperator{\Rhomo}{\typesetCategory{\GRPH}_{R-homo}}

\DeclareMathOperator{\nat}{\typesetCategory{Nat}}
\DeclareMathOperator{\HGr}{\typesetCategory{HGr}}

\definecolor{gold}{RGB}{255,215,0}
\definecolor{softBlack}{RGB}{45, 47, 49}
\definecolor{creamWhite}{RGB}{245,244,241}
\definecolor{softGray}{RGB}{220, 216, 214}
\definecolor{brick}{RGB}{232, 48, 48}

\usepackage{enumitem}

\title{Spined categories: generalizing tree-width beyond graphs}
\author{Benjamin Merlin Bumpus \thanks{Mathematics and Computer Science, TU Eindhoven, GroeneLoper MetaForum Building PO Box 513, Eindhoven, 5600 MB, North Brabant, Netherlands. Supported both by an EPSRC studentship, and by the European Research Council (ERC) under the European Union's Horizon 2020 research and innovation programme (grant agreement No 803421, ReduceSearch)}   ~ \&  Zoltan A. Kocsis \thanks{CSIRO Data61, University of New South Wales, Kensington NSW 2052, Australia.}}
\date{\today}

\begin{document}

\maketitle

\begin{abstract}
Tree-width is an invaluable tool for computational problems on graphs. But often one would like to compute on other kinds of objects (e.g. decorated graphs or even algebraic structures) where there is no known tree-width analogue. Here we define an abstract analogue of tree-width which provides a uniform definition of various tree-width-like invariants including graph tree-width, hypergraph tree-width, complemented tree-width and even new constructions such as the tree-width of modular quotients. We obtain this generalization by developing a general theory of categories that admit abstract analogues of both tree decompositions and tree-width; we call these pseudo-chordal completions and the triangulation functor respectively.
\end{abstract}

\section{Introduction}
Divide and conquer algorithms are nearly ubiquitous in computer science. Indeed, already in the 1980s, Johnson~\cite{JOHNSON1985434} compiled a list of graph classes satisfying the following properties: 
\begin{itemize}
    \item their members admit recursive decompositions into smaller and simpler pieces, 
    \item many $\textsc{NP}$-hard problems can be solved in polynomial time by using these structural decompositions as data structures for dynamic programming.
\end{itemize}
Some examples from Johnson's list are: trees, partial $k$-trees, chordal graphs, series parallel graphs, split graphs and co-graphs. 

Today we are equipped with a vast literature on how to go about decomposing graphs into smaller parts and how to exploit these decompositions for algorithmic purposes. Indeed, many important graph classes which admit powerful algorithmic meta-theorems are described by measuring the \textit{width} of smallest `structural decompositions' of their members (e.g. classes of bounded tree-width or clique-width). 

Due to a plethora of algorithmic and theoretical applications, significant amounts of research effort are spent on introducing and studying new sorts of decompositions (along with corresponding width measures) that capture specific classes of objects and/or structural correspondences. These include e.g.~\emph{modular decompositions}~\cite{HABIB201041}, \emph{partitive families}~\cite{CHEIN198135, cunningham_edmonds_1980}, \emph{clique-width decomposition trees}~\cite{courcelle1993, COURCELLE199687}, \emph{branch decompositions}~\cite{robertsonX} and \emph{rank decompositions}~\cite{oum2006approximating}.

Featuring most prominently among these graph parameters is indubitably \textit{tree-width}: it has played a key part in the proof of the celebrated Robertson-Seymour graph minor theorem~\cite{RobertsonXX} and it is a powerful hammer in the parameterized complexity toolbox~\cite{flum2006parameterized, cygan2015parameterized}.

\paragraph{This paper} is motivated by the long-term ambition of extending the aforementioned decomposition methods to other kinds of mathematical structures. The specific question which we address is that of finding a uniform and abstract definition of both tree decompositions and tree-width which can then be applied to structures other than finite simple graphs. Our answer to this problem requires a change in perspective: we shift from a graph-theoretic point of view to one which is category-theoretic. This approach leads us to obtain a vast generalization of tree-width -- called the \emph{triangulation functor} -- which is defined for all objects of special kinds of categories which we call \emph{spined categories} (we delay a detailed overview our our contributions to end of this section). 

The question of generalizing notions such as tree-width to new settings is as natural as it is challenging. Indeed it is often the case that in practice one needs to compute not on graphs, but on other kinds of objects (such as decorated graphs or perhaps algebraic structures). Unfortunately, we so far do not have a wealth of decomposition methods for such objects and neither do we have structured, formulaic ways of finding such generalizations. 

The difficulty of generalization to which we refer already rears its head when one tries to lift the definition of tree-width from graphs to directed graphs. Indeed, this has been a challenging (and ongoing) research question that has captivated the research community since the 1990s~\cite{kreutzer2018} and which has led to the definition of a myriad of subtly different directed analogues of tree-width~\cite{dtw, berwanger2012dag, hunter2008digraph, safari2005d, kreutzer2018}.

When one is interested in obtaining tree-width analogues even further afield from the familiar settings of graphs and directed graphs, further challenges arise due to the fact that most of the aforementioned decomposition methods are defined explicitly in terms of the internal structure of the decomposed object (think, for example of the definition of tree-width or of clique-width decomposition trees, which use a formal grammar to specify how to construct a given graph from smaller ones). This makes it arduous to generalize a given notion of decomposition to different classes of objects especially if we wish to place as few restrictions as possible on such classes (for example, it could be that the objects we wish to decompose might not even come equipped with obvious analogues of the notions of 'vertex' or 'edge' or 'connectivity' which, of course, tend to feature prominently in graph-theoretic definitions). 

One of the points that we make in this paper is that some of the difficulties of transferring a given decomposition notion to a more general setting can sidestepped by finding a characteristic property which: (1) defines the decomposition of an object $X$ independently of the internal structure of $X$ and (2) which is formulated purely in terms of the \textit{category} inhabited by the objects we wish to decompose. Such category-theoretic characterizations have already proven successful in other fields (see e.g. Leister's work on categorial characterizations of ultraproducts~\cite{leinster2013codensity} and more recently Lebesgue integration~\cite{leinsterLebesgue}).

\paragraph{Contributions.} 
We introduce \textit{Spined categories}. These are categories $\mathcal{C}$ equipped with sufficient additional structure to admit both
\begin{enumerate}
    \item an abstract analogue of tree decompositions (which we call \textit{pseudo-chordal completions}) and
    \item a categorial generalization of tree-width (we call this the \textit{triangulation functor}) exhibited as as a special kind of functor (which we call \emph{S-functors}) from $\mathcal{C}$ to the poset of natural numbers.
\end{enumerate} 
When instantiated in the category of graphs and subgraphs, our construction agrees with the graph-thereotic one: denoting by $\Delta_{\GRPH}$ the triangulation functor (i.e. our abstract analogue of tree-width) in the category of graphs, we have $\Delta_{\GRPH}(G) = \tw(G) + 1$ for any graph $G$ (furthermore, this statement holds true even when we replace 'graph' by 'hypergraph' throughout). 

The list of examples of spined categories is not restricted to just these familiar cases; other examples include categories having as objects:  natural numbers, posets and even vertex-labeling functions. 

Our abstract analogue of tree-width is defined via the notion of an \emph{S-functor}. This constitutes a vast generalization of Halin's S-functions~\cite{halin1976s} (which are graph invariants of the form $(G \in \text{Graphs}) \mapsto (n \in \mathbb{N})$ which share several properties with the Hadwiger number, modified chromatic number\footnote{The maximum of the chromatic numbers over all minors~\cite{halin1976s}} and the modified connectivity number\footnote{One plus the maximum of the connectivity number over all minors~\cite{halin1976s}}). Indeed our main theorem (Theorem~\ref{thm:triang-functor-is-S-functor}) can also be seen as a generalization of Halin's definition of tree-width as the maximal S-function out of the (point-wise ordered) poset of all S-functions.

Our construction of triangulation functors is uniform on all spined categories and thus allows one to define new tree-width-like parameters (such as widths for new types of combinatorial objects, or notions of graph width that respect stricter notions of structural correspondence than ordinary graph isomorphism) simply by collecting the relevant objects into a spined category.

\paragraph{Outline.} To accommodate readers from different backgrounds, Section~\ref{sec:background} consists of short review of the graph- and category-theoretic background required for this paper. In Section~\ref{sec:defining_spined_categories} we introduce \textit{spined categories} (Definition~\ref{def:spined_cat}) and the corresponding notion of morphisms, \textit{spinal functors} (Definition \ref{def:spinal-functor}). Section \ref{sec:triangulation_number_results} contains the proof of our main result (Theorem~\ref{thm:triang-functor-is-S-functor}) on the existence of spinal functors called \textit{triangulation functors} which generalize tree-width-like invariants used in combinatorics. In Section \ref{sec:new_from_old} we describe a way of constructing new spined categories from previously known ones and illustrate the applications of such constructions with some examples. Section~\ref{sec:discussion} briefly discusses open questions and directions for future research.

\paragraph{Acknowledgements:} we would like to thank Steven Noble and Gramoz Goranci for their detailed feedback on the preliminary version of this article; furthermore we extend our thanks to Karl Heuer, Kitty Meeks, Ambroise Lafont, Johannes Carmesin and Bart Jansen for their comments, suggestions and discussions which helped in consolidating and maturing the ideas in this paper as well as improving their exposition.

\section{Background}\label{sec:background}
For any graph-theoretic notation not defined here, we refer the reader to Diestel's textbook \cite{Diestel2010GraphTheory} while, for category-theoretic terminology not mentioned here, we refer to Awodey's textbook~\cite{awodey2010category}. For a more gentle introduction to the background and most (but not all) of the results of this paper, we direct the reader to thesis by Bumpus~\cite{bumpus2021generalizing}. Finally we note that an extended abstract on this topic appeared at the Applied Category Theory conference ACT2021~\cite{our-extended-abstract-ACT-2021}.

The class $\mathcal{G}$ consists of all finite graphs that have no loops or parallel edges. By "graph" we always mean an element of $\mathcal{G}$ (note that this clashes with the common convention in category theory, which takes graphs to be \textit{reflexive}). Throughout, for any natural number $n$, let $[n]$ denote the set $\{1, \ldots, n\}$. We write $K_n$ (resp. $\overline{K_n}$) for the complete graph (resp. discrete graph) on the set $[n]$. We denote the \emph{disjoint union} of two sets $A$ and $B$ by $A \uplus B$. For graphs $G$ and $H$, we denote by $G \uplus H$ and $G \cap H$ respectively the graphs $(V(G) \uplus V(H), E(G), \uplus E(H))$ and $(V(G) \cap V(H), E(G) \cap E(H))$. We call a vertex $v$ of a graph $G$ an \emph{apex} if it is adjacent to every other vertex in $G$. We denote by $G \star v$ the operation of adjoining a new apex vertex $v$ to $G$.

A \textit{circuit of length}~$n$ in the simple graph $G$ is a finite sequence $(e_1,\dots,e_n)$ of edges of $G$ such that consecutive edges share an endpoint, as do $e_1$ and $e_n$. A simple graph that contains no circuits is called a \textit{tree}.

A \emph{graph homomorphism} from a graph $G$ to a graph $H$ is a function $h: V(G) \to V(H)$ such that $h(x)h(y) \in E(H)$ whenever $xy \in E(G)$. Note that, if $G$ is a subgraph of $H$, then there is an injective graph homomorphism $\phi: G \to H$ which witnesses this fact. 

\subsection{Tree-width}
Intuitively, the tree-width function, $\tw: \mathcal{G} \to \mathbb{N}$ measures how far a given graph is from being a tree. For example, edge-less graphs have tree-width $0$, forests with at least one edge have tree-width $1$ and, for $n > 1$, $n$-vertex cliques have tree-width $n-1$. Tree-width was introduced independently by many authors \cite{Bertele1972NonserialProgramming, halin1976s, RobertsonII} and thus has many equivalent definitions; the most common definition makes use of the related concept of a tree-decomposition. Here we give its definition for hypergraphs \cite{ADLER20072167}.

\begin{definition}\cite{ADLER20072167}\label{def:treedecompositionDefinition}
The pair $(T, (B_t)_{t \in V(T)})$ is a \emph{tree decomposition} of a hypergraph $H$ if $(B_t)_{t \in V(T)})$ is a sequence of subsets (called \emph{bags}) of $V(H)$ indexed by the nodes of the tree $T$ such that:
\begin{enumerate}[label=\textbf{( T\arabic*)}]
\item for every hyper-edge $F$ of $H$, there is a node $t \in V(T)$ such that $F \subseteq B_t$,
\item for every $x \in V(H)$, the set $V_{(T,x)} := \{t \in V(T): x \in B_t\}$ induces a connected subgraph in $T$ (in particular $V_{(T,x)}$ is not empty).
\end{enumerate}
\end{definition}

\noindent The \emph{width} of a tree decomposition $(T, (V_t)_{t \in T})$ of the hypergraph $H$ is defined as one less than the maximum of the cardinalities of its bags. The \emph{tree-width} $\tw(H)$ of $H$ is the minimum possible width of any tree decomposition of $H$. (The definition of tree decomposition and tree-width follow for simple graphs by viewing them as $2$-uniform hypergraphs.)

Halin~\cite{halin1976s} provides an alternative characterization of tree-width as a maximal element in a class of functions called \emph{S-functions}. These are mappings from finite graphs to $\mathbb{N}$ satisfying a set of common properties fulfilled by the \emph{chromatic number}, \emph{vertex-connectivity number} and the \emph{Hadwiger number}. In order to define S-functions, we first recall the concept of an $H$-sum of two graphs.

\begin{definition}\label{def:H_sum}
Given two graphs $G_1$ and $G_2$ and a subgraph $H$ of both of them, the $H$-sum of $G_1$ and $G_2$ is the graph $G_1 \#_H G_2$ obtained by identifying the vertices of $H$ in $G_1$ to the vertices of $H$ in $G_2$ and removing any parallel edges. Formally, given injective homorphisms $h_i: H \to G_i$ witnessing that $H$ is a subgraph in $G_1$ and $G_2$, the graph $G_1 \#_H G_2$ is defined as
$$G_1 \#_H G_2 := \bigl ( V(G_1) \uplus V(G_2) /_{h_1 = h_2}, E(G_1) \uplus E(G_2) /_\sim \bigr )$$
where edges $wx$ and $yz$ are related under $\sim$ if $\{h_1(w), h_1(x)\} = \{h_2(y), h_2(z)\}$.
\end{definition}

\noindent Given the definition of $H$-sum, we can now recall Halin's definition of S-function.  

\begin{definition}[\cite{halin1976s}]\label{def:halin_S-func}
A function $f: \mathcal{G} \rightarrow \mathbb{N}$ is called an \emph{S-function} if it satisfies the following four properties:
\begin{enumerate}[label=\textbf{(H\arabic*)}]
    \item $f(K_0) = 0$ ($K_0$ is the empty graph)\label{axiom:H1}
    \item if $H$ is a minor of $G$, then $f(H) \leq f(G)$ (minor isotonicity) \label{axiom:H2}
    \item $f(G \star v) = 1 + f(G)$ (distributivity over adding an apex) \label{axiom:H3}
    \item for each $n \in \mathbb{N}$, $G = G_1 \#_{K_n} G_2$ implies that $f(G) = \max_{i \in \{1,2\}}f(G_i)$ (distributivity over clique-sum). \label{axiom:H4}
\end{enumerate}
\end{definition}

\begin{theorem}[\cite{halin1976s}]\label{thm:Halin_tw}
The set of all S-functions forms a complete distributive lattice when equipped with the pointwise ordering. Furthermore, the function $G \mapsto \tw(G) + 1$ is maximal in this lattice.
\end{theorem}

\subsection{Category-theoretic preliminaries}
Our generalization of Halin's characterization of tree-width relies on some standard category-theoretic tools. To keep the presentation self-contained, we recall the definitions of all relevant concepts here. 

Throughout we let $\mathbb{N}_{\leq}$ denote the ordered set of natural numbers (under the usual ordering $\leq$), regarded as a category the obvious way. Similarly, $\mathbb{N}_{=}$ denotes the discrete category whose objects are natural numbers. We write $\GRPH_{homo}$ for the category having finite simple graphs as objects and graph homomorphisms as arrows.

We call a morphism (or arrow) $f: A \rightarrow B$ in a category $\mathcal{C}$ a \textit{monomorphism in $\mathcal{C}$} (or a \textit{monic arrow in $\mathcal{C}$}) if, given any two arrows $x, y : Z \to A$, we have $f \circ x = f \circ y$ implies $x = y$. Throughout this text the notation $f: A \hookrightarrow B$ always denotes a monomorphism from $A$ to $B$. Given a category $\mathcal{C}$, we let $\mono(\mathcal{C})$ denote the subcategory of $\mathcal{C}$ given by all the monic arrows of $\mathcal{C}$ (note that this differs from the standard usage, where $\mono(\mathcal{C})$ denotes a specific subcategory of the arrow category $\mathrm{Arr}(C)$ of $\mathcal{C}$ instead).

\begin{definition}
A \textit{functor} $F$ between the categories $\mathcal{C}$ to $\mathcal{D}$ is a mapping that associates
\begin{itemize}
    \item to every object $W$ in $\mathcal{C}$ an object $F[W]$ in $\mathcal{D}$
    \item to every arrow $f: X \to Y$ in $\mathcal{C}$ an arrow $F[f]: F[X] \to F[Y]$ in $\mathcal{D}$
\end{itemize}
while preserving identity and compositions, i.e. 
\begin{itemize}
    \item $F(\id_X) = \id_{F[X]}$ for every object $X$ in $\mathcal{C}$, and
    \item $F[g \circ f] = F[g] \circ F[f]$ for all arrows $f: X \to Y$, $g: Y \to Z$ in $\mathcal{C}$.
\end{itemize}
A \textit{diagram of shape} $J$ in a category $\mathcal{C}$ is a functor from $J$ to $\mathcal{C}$.
\end{definition}

\noindent We call a diagram of shape \begin{tikzcd} G & A \arrow[l, "g"] \arrow[r, "h"'] & H\end{tikzcd} in the category $\mathcal{C}$ a \emph{span in $\mathcal{C}$}; similarly, we call \begin{tikzcd}  G \arrow[r, "g"] & A & \arrow[l, "h"'] H\end{tikzcd} a \emph{cospan}. A \emph{monic} (co)span is a (co)span consisting of monic arrows. 

\begin{definition}\label{def:pushout}
Consider a span \begin{tikzcd} G_1 & H \arrow[l, "g_1"] \arrow[r, "g_2"'] & G_2\end{tikzcd} in a  category $\mathcal{C}$. The cospan \begin{tikzcd} G_1 \arrow[r, "g_1^+"] & G_1 +_{H} G_2 & \arrow[l, "g_2^+"'] G_2\end{tikzcd} is a \textit{pushout of $g_1$ and $g_2$ in $\mathcal{C}$} if 
\begin{enumerate}
    \item $g_1^+ \circ g_1 = g_2^+ \circ g_2$, and
    \item for any cospan \begin{tikzcd} G_1 \arrow[r, "z_1"] & Z & \arrow[l, "z_2"'] G_2\end{tikzcd} such that $z_1 \circ g_1 = z_2 \circ g_2$ (a \textit{cocone} of the span) we can find a \emph{unique} morphism $m: G_1 +_{H} G_2 \rightarrow Z$ such that $m \circ g_1^+ = z_1$ and $m \circ g_2^+ = z_2$.
\end{enumerate}
We call $G_1 +_{H} G_2$ the \emph{pushout object of $g_1$ and $g_2$}.
\end{definition} 

Pushouts in $\GRPH_{homo}$ allow us to recover the definition of an $H$-sum of graphs (recall Definition \ref{def:H_sum}).

\begin{proposition}\label{prop:grph-pushout-as-clique-sum}
Every monic span in $\GRPH_{homo}$ has a pushout. In particular, the pushout of a monic span \begin{tikzcd} G_1 & H \arrow[l, hook', "g_1"] \arrow[r, hook, "g_2"'] & G_2\end{tikzcd} is the graph $G_1 \#_H G_2$ given by the $H$-sum of $G_1$ and $G_2$ along $H$. 
\begin{proof}
Take the obvious inclusion maps as $\iota_1: G_1 \hookrightarrow G_1 \#_H G_2$ and $\iota_2: G_2 \hookrightarrow G_1 \#_H G_2$. We clearly have $\iota_1 \circ g_1 = \iota_2 \circ g_2$. Now consider any other cospan \begin{tikzcd} G_1 \arrow[r, "z_1"] & Z & \arrow[l, "z_2"'] G_2\end{tikzcd} satisfying the equality $z_1 \circ g_1 = z_2 \circ g_2$. Define the map $m: G_1 \#_H G_2 \rightarrow Z$ on the vertices of $G_1 \#_H G_2$ via the equation
\begin{align*}
  m(v) = 
  \begin{cases}
    z_1(v) & \text{if $v \in G_1$,} \\
    z_2(v) & \text{otherwise.}
  \end{cases}
\end{align*}
Notice that $m$ is well-defined since if $v \in V(G_1) \cap V(G_2)$, then $z_1(v) = z_1(g_1(v)) = z_2(g_2(v)) = z_2(v)$. 

We check that $m \circ \iota_1 = z_1$. By extensionality, it suffices to prove $m(\iota_1(x)) = z_1(x)$ for an arbitrary vertex $x$ of $G$. Since $\iota_1(x) = x$ and $x \in G$, the first clause of the definition applies, and we have $m(\iota_1(x)) = m(x) = z_1(x)$. A similar proof allows us to conclude $m \circ \iota_2 = z_2$. The uniqueness of $m$ follows immediately.
\end{proof}
\end{proposition}

We cannot generalize Proposition~\ref{prop:grph-pushout-as-clique-sum} much further, since the pushout of an arbitrary pair $(i: D \rightarrow G,j: D \rightarrow H)$ need not exist in $\GRPH_{homo}$. Indeed, taking the obvious injection $i: \overline{K_2} \rightarrow K_2$ and the unique map $j: \overline{K}_2 \rightarrow K_1$, we see that no object $Z$ and map $z_1: K_2 \rightarrow Z$ can ever satisfy $z_1 \circ i = z_2 \circ j$, since the image of the right-hand side always consists of a single vertex, while the image of the left-hand side necessarily contains an edge.

\section{Spined Categories and S-functors}\label{sec:defining_spined_categories}
Here we introduce \emph{spined categories}, categories with sufficient extra structure to admit a categorial generalization of the graph-theoretic notion of tree-width (the \textit{triangulation functor}, constructed in Section \ref{sec:triangulation_number_results}).

Spined categories come equipped with a notion of \textit{proxy pushout}, whose role is largely analogous to that of the clique-sum operation in Halin's definition of S-functions (Definition \ref{def:halin_S-func}). Proxy pushouts are similar to, but significantly less restrictive than ordinary pushouts: in fact, pushouts always give rise to proxy pushouts (Proposition \ref{prop:pushouts-as-proxy-pushouts}), but the converse does not hold. The role of cliques themselves is played by the members of a distinguished sequence of objects, called the \textit{spine}.

Among the structure-preserving functors between spined categories, we find abstract, functorial counterparts to Halin's S-functions: these are the \textit{S-functors} of Definition \ref{def:S_functor}. We shall see that S-functors are in fact more general than Halin's notion, even in the case of simple graphs. While every S-function yields an S-functor over the category $\GRPH_{mono}$ (Proposition \ref{prop:S_functions_are_S_functors}), the converse is not true.

\begin{definition}\label{def:spined_cat}
A \textit{spined category} consists of a category $\mathcal{C}$ equipped with the following additional structure:
\begin{itemize}
    \item a functor $\Omega: \mathbb{N}_= \rightarrow \mathcal{C}$ called the \textit{spine} of $\mathcal{C}$,
    \item an operation $\mathfrak{P}$ (called the \textit{proxy pushout}) that assigns to each span of the form
\begin{tikzcd}
G & \Omega_n \arrow[l, "g"'] \arrow[r, "h"] & H
\end{tikzcd}
    in $\mathcal{C}$ a distinguished cocone
\begin{tikzcd}
G \arrow[r, "{\mathfrak{P}(g,h)_g}"] & {\mathfrak{P}(g,h)} & H \arrow[l, "{\mathfrak{P}(g,h)_h}"']
\end{tikzcd}
\end{itemize}
subject to the following conditions:
\begin{enumerate}[label=\textbf{SC\arabic*}]
    \item For every object $X$ of $\mathcal{C}$ we can find a morphism $x: X \rightarrow \Omega_n$ for some $n \in \mathbb{N}$.\label{property:SC1}
    \item For any cocone
\begin{tikzcd}
G & \Omega_n \arrow[l, "g"'] \arrow[r, "h"] & H
\end{tikzcd}
    and any pair of morphisms $g': G \rightarrow G'$ and $h': H \rightarrow H'$ we can find a \emph{unique} morphism $(g',h'): \mathfrak{P}(g,h) \rightarrow \mathfrak{P}(g' \circ g, h' \circ h)$ making the following diagram commute:
\begin{center}
\begin{tikzcd}
\Omega_n \arrow[d, "h"'] \arrow[r, "g"]                            & G \arrow[r, "g'"] \arrow[d, "{\mathfrak{P}(g,h)_g}"] & G' \arrow[dd, "{\mathfrak{P}(g'\circ g,h'\circ h)_{g' \circ g}}"] \\
H \arrow[d, "h'"'] \arrow[r, "{\mathfrak{P}(g,h)_h}"']             & {\mathfrak{P}(g,h)} \arrow[rd, "{(g',h')}", dashed]  &                                                                   \\
H' \arrow[rr, "{\mathfrak{P}(g'\circ g,h'\circ h)_{h' \circ h}}"'] &                                                      & {\mathfrak{P}(g' \circ g,h' \circ h)}                            
\end{tikzcd}
\end{center}\label{property:SC2}
\end{enumerate}
\end{definition}

One could define \textit{proxy pullbacks} dually to proxy pushouts. As the name suggests, categories with enough pushouts (pullbacks) always have proxy pushouts (proxy pullbacks). This observation gives rise to many examples of spined categories.

\begin{proposition}\label{prop:pushouts-as-proxy-pushouts}
Take a category $\mathcal{C}$ equipped with functor $\Omega: \mathbb{N}_= \rightarrow \mathcal{C}$ such that the following hold:
\begin{enumerate}
    \item for any object $X$ of $\mathcal{C}$ there is some $n \in \mathbb{N}$ and morphism $x: X \rightarrow \Omega_n$, and
    \item every span of the form
\begin{tikzcd}
G & \Omega_n \arrow[l, "g"'] \arrow[r, "h"] & H
\end{tikzcd}
    has a pushout in $\mathcal{C}$.
\end{enumerate}
The map $\mathfrak{P}$ that assigns to every span
\begin{tikzcd}
G & \Omega_n \arrow[l, "g"'] \arrow[r, "h"] & H
\end{tikzcd}
its pushout square turns $\mathcal{C}$ into a spined category.
\begin{proof}
We only have to verify Property~\ref{property:SC2}. Consider the diagram
\begin{center}
\begin{tikzcd}
\Omega_n \arrow[d, "h"'] \arrow[r, "g"]  & G \arrow[r, "g'"] \arrow[d, "\iota_G"]         & G' \arrow[dd, "\iota_G'"] \\
H \arrow[d, "h'"'] \arrow[r, "\iota_H"'] & G+_{\Omega_n}H \arrow[rd, "{ }", dotted] &                           \\
H' \arrow[rr, "\iota_H'"']               &                                                & G' +_{\Omega_n} H'       
\end{tikzcd}
\end{center}
We have to exhibit the unique dotted morphism $G +_{\Omega_n} H \rightarrow G' +_{\Omega_n} H'$ making this diagram commute. Notice that the arrows $\iota_G' \circ g'$ and $\iota_H' \circ h'$ form a cocone of the span of $g,h$. Since the pushout of $g$ and $h$ is universal among such cocones, the existence and uniqueness of the required morphism $G +_{\Omega_n} H \rightarrow G' +_{\Omega_n} H'$ follows.
\end{proof}
\end{proposition}

Since pushouts in poset categories are given by least upper bounds, Proposition~\ref{prop:pushouts-as-proxy-pushouts} allows us to construct a simple (but important) first example of a spined category.

\begin{example}\label{example:spined-category-Nat}
Let $\leq$ denote the usual ordering on the natural numbers. The poset $\mathbb{N}_\leq$, when equipped with the spine $\Omega_n = n$ (and maxima as proxy pushouts) constitutes a spined category denoted $\nat$.
\end{example}

Combining Propositions~\ref{prop:grph-pushout-as-clique-sum}~and~\ref{prop:pushouts-as-proxy-pushouts} gives us a first example of a "combinatorial" spined category, the category $\GRPH_{mono}$ which has graphs as objects and injective graph homomorphisms as arrows. First consider a span of the form
\begin{tikzcd} A & X \arrow[l, hook'] \arrow[r, hook] & B\end{tikzcd}
in $\GRPH_{homo}$. Notice that all arrows are monic in the corresponding pushout square. However, given a cocone
\begin{tikzcd} A \arrow[r, "a"'] & Z & B \arrow[l, "b"] \end{tikzcd}
the pushout morphism $A +_X B \rightarrow Z$ can fail to be a monomorphism (for instance in the case where the images of $a$ and $b$ have non-empty intersection). It follows that the clique sum does not give rise to pushouts in the category $\GRPH_{mono}$. Nonetheless, the category satisfies Property~\ref{property:SC2}, so the lack of pushouts does not stop us from constructing a spined category.

\begin{proposition}\label{prop:proxy-pushouts-in-grph-mono}
The category $\GRPH_{mono}$, equipped with the spine $n \mapsto K_n$ and clique sums as proxy pushouts forms a spined category.
\begin{proof}
Property~\ref{property:SC1} is evident, but we need to verify Property~\ref{property:SC2}. Consider the diagram
\begin{center}
\begin{tikzcd}
\Omega_n \arrow[d, "h"', hook] \arrow[r, "g", hook] & G \arrow[r, "g'", hook] \arrow[d, "\iota_G", hook] & G' \arrow[dd, "\iota_G'", hook] \\
H \arrow[d, "h'"', hook] \arrow[r, "\iota_H"', hook]      & G\#_{\Omega_n}H \arrow[rd, "!p", dashed]     &                           \\
H' \arrow[rr, "\iota_H'"', hook]                          &                                              & G' \#_{\Omega_n} H'      
\end{tikzcd}
\end{center}
in $\GRPH_{homo}$. Notice that the arrows $\iota_G, \iota_G', \iota_H, \iota_H'$ are all monic. We have to establish that the morphism $p: G\#_{\Omega_n}H \rightarrow G' \#_{\Omega_n} H'$ (which is unique since it is a pushout arrow in $\GRPH_{homo}$) is monic as well. Note that $p$ maps any vertex $x$ in $G\#_{\Omega_n}H$ to $(\iota'_G \circ g')(x)$ if $x$ is in $G$ and to $(\iota'_H \circ h')(x)$ otherwise. Thus, since $V(G') \cap V(H') = V(G) \cap V(H)$, we have that, for any $x$ and $y$ in $V(G\#_{\Omega_n}H)$, if $p(x) = p(y)$ then $x = y$. Thus $p$ is injective (i.e. it is monic and hence it is in $\GRPH_{mono}$).
\end{proof}
\end{proposition}

\noindent We will encounter further examples of spined categories below, including:

\begin{enumerate}
    \item the poset of natural numbers under the divisibility relation (Proposition \ref{prop:generalized-clique-number-not-S-functor}),
    \item the category of posets (Proposition \ref{prop:poset-lacks-sfunctors}),
    \item the category of hypergraphs (Theorem \ref{thm:hypergraphs}),
    \item the category of vertex-labelings of graphs (Examples \ref{example:modular_tree_width} and \ref{example:chromatic_tree_width}).
\end{enumerate}

Now we introduce the notion of a \emph{spinal functor} as the obvious notion of morphism between two spined categories.

\begin{definition}\label{def:spinal-functor}
Consider spined categories $(\mathcal{C},\Omega^\mathcal{C}, \mathfrak{P}^\mathcal{C})$ and $(\mathcal{D},\Omega^\mathcal{D}, \mathfrak{P}^\mathcal{D})$. We call a functor $F: \mathcal{C} \rightarrow \mathcal{D}$ a \textit{spinal functor} if it
\begin{enumerate}[label=\textbf{SF\arabic*}]
    \item \textit{preserves the spine}, i.e. $F \circ \Omega^\mathcal{C} = \Omega^\mathcal{D}$, and \label{property:SF1}
    \item \textit{preserves proxy pushouts}, i.e. given a proxy pushout square
\begin{center}
\begin{tikzcd}
\Omega_n \arrow[d, "h"'] \arrow[r, "g"] & G \arrow[d, "{\mathfrak{P}^\mathcal{C}(g,h)_g}"] \\
H \arrow[r, "{\mathfrak{P}^\mathcal{C}(g,h)_h}"']   & {\mathfrak{P}^\mathcal{C}(g,h)}                 
\end{tikzcd}
\end{center}
    in the category $\mathcal{C}$, the image
\begin{center}
\begin{tikzcd}
\Omega_n \arrow[d, "Fh"'] \arrow[r, "Fg"]   & {F[G]} \arrow[d, "{F\mathfrak{P}^\mathcal{C}(g,h)_g}"] \\
{F[H]} \arrow[r, "{F\mathfrak{P}^\mathcal{C}(g,h)_h}"'] & {F[\mathfrak{P}^\mathcal{C}(g,h)]}
\end{tikzcd}
\end{center}
    forms a proxy pushout square in $\mathcal{D}$. Equivalently, $F[\mathfrak{P}^\mathcal{C}(g,h)] = \mathfrak{P}^\mathcal{D}(Fg,Fh)$, $F\mathfrak{P}^\mathcal{C}(g,h)_g = \mathfrak{P}^\mathcal{D}(Fg,Fh)_{Fg}$ and $F\mathfrak{P}^\mathcal{C}(g,h)_h = \mathfrak{P}^\mathcal{D}(Fg,Fh)_{Fh}$ all hold.\label{property:SF2}
\end{enumerate}
\end{definition}

\noindent Recall the spined category $\mathbf{Nat}$ of Example~\ref{example:spined-category-Nat}. Using spinal functors to $\nat$, we obtain the following categorial counterparts to Halin's S-functions.

\begin{definition}\label{def:S_functor}
An \emph{S-functor} over the spined category $\mathcal{C}$ is a spinal functor ${F: \mathcal{C} \rightarrow \nat}$.
\end{definition}

Proxy pushouts in $\GRPH_{mono}$ are given by clique sums over complete graphs, while pushouts in $\nat$ are given by maxima. Consequently, given an S-functor $F: \GRPH_{mono} \rightarrow \nat$, Property~\ref{property:SF2} reduces to the equality $F[G \#_{K_n} H] = \max\{F[G],F[H]\}$ (cf. Property~\ref{axiom:H4} of Halin's S-functions).

\begin{proposition}\label{prop:S_functions_are_S_functors}
Every S-function $f: \mathcal{G} \rightarrow \mathbb{N}$ gives rise to an S-functor $F$ satisfying $F[X] = f(X)$ for all objects $X$ of $\GRPH_{mono}$.
\begin{proof}
Take an S-function $f: \mathcal{G} \rightarrow \mathbb{N}$. Take a morphism $f: X\rightarrow Y$ in $\mathcal{C}$. Since $f$ is a graph monomorphism, $X$ is isomorphic to a subgraph of $Y$, and is therefore a (trivial) minor of $Y$. Thus, $f(X) \leq f(Y)$ holds by Property~\ref{axiom:H2}. It follows that the map $F$ defined by the equations $F[X] = f(X)$ and $Ff = (F[X] \leq F[Y])$ for each pair of objects $X,Y$ and each morphism $f: X \rightarrow Y$ constitutes a functor from $\GRPH_{mono}$ to the poset category~$\mathbb{N}_\leq$.

We show that $F$ preserves the spine inductively, by proving $F[K_n] = f(K_n) = n$ for all $n \in \mathbb{N}$:
\begin{itemize}
    \item \textbf{Base case:} We have $F[K_0] = 0$ by Property~\ref{axiom:H1}.
    \item \textbf{Inductive case:} Assume that $F[K_n] = f(K_n) = n$. Since $K_n \star v = K_{n+1}$, we have $F[K_{n+1}] = f(K_{n+1}) = f(K_n \star v) = 1 + f(K_n) = 1 + n$ by Property~\ref{axiom:H3}.
\end{itemize}
The preservation of proxy pushouts follows immediately by Property~\ref{axiom:H4}. Hence $F$ is a spinal functor as we claimed.
\end{proof}
\end{proposition}

We note, however that the converse of Proposition \ref{prop:S_functions_are_S_functors} does not hold (not even in $\GRPH_{mono}$). To see this, note that while the clique number is an $S$-functor in $\GRPH_{mono}$, it may increase when taking minors. Thus the clique number does not satisfy Property \ref{axiom:H2} and hence it is not an $S$-function. 

Using the natural indexing on the spine given by the functor $\Omega: \mathbb{N}_= \rightarrow \mathcal{C}$, we can associate the following numerical invariants to each object of the spined category $\mathcal{C}$. 

\begin{definition}\label{def:order}
Take a spined category $(\mathcal{C}, \Omega, \mathfrak{P})$ and an object $X \in \mathcal{C}$. We define the \textit{order} $|X|$ of the object $X$ as the least $n \in \mathbb{N}$ such that $\mathcal{C}$ has a morphism $X \rightarrow \Omega_n$. Similarly, we define the \textit{generalized clique number} $\omega(X)$ as the largest $n \in \mathbb{N}$ for which~$\mathcal{C}$ contains a morphism $\Omega_n \rightarrow X$ (whenever such $n$ exists).
\end{definition}

It's clear that a spined category $(\mathcal{C}, \Omega_n, \mathfrak{P})$ where $|\Omega_n| < n$ (resp.~ $\omega(\Omega_n) > n$) does not admit any S-functors since it then would be impossible for any candidate $S$-functor to preserve the spine. In particular there are no S-functors defined on the category $\GRPH_{homo}$. However, S-functors may fail to exist even if $|\Omega_n| = n$ ($\omega(\Omega_n) = n$). We construct such an example below.

\begin{proposition}\label{prop:poset-lacks-sfunctors}
There exist spined categories $(\mathcal{C}, \Omega, \mathfrak{P})$ satisfying $|\Omega_n| = n = \omega(\Omega_n)$  that do not admit any S-functors.
\begin{proof}
Consider the category $\mathbf{Poset}_{mono}$ which has finite posets as objects and order-preserving injections as morphisms. Let $\Omega_n$ denote set $\SetComp{m \in \mathbb{N}}{m \leq n}$ under its usual linear ordering, and let $\mathfrak{P}$ assign to each span of the form
\begin{tikzcd}
G & \Omega_n \arrow[l, "g"'] \arrow[r, "h"] & H
\end{tikzcd}
the pushout $G +_{\Omega_n} H$ of the span in $\mathbf{Poset}_{homo}$ (the category of posets is cocomplete~\cite{awodey2010category}, so in particular it has all pushouts). We will show that the structure $(\mathbf{Poset}_{mono}, \Omega, \mathfrak{P})$ forms a spined category that does not admit any S-functors. 

Take any poset $P$ on $n$ elements and note that there is a monomorphism from $P$ to $L_n$. This verifies Property~\ref{property:SC1}. For Property~\ref{property:SC2} consider the following diagram.
\begin{center}
\begin{tikzcd}
\Omega_n \arrow[d, "h"', hook] \arrow[r, "g", hook] & G \arrow[r, "g'", hook] \arrow[d, "\iota_G", hook] & G' \arrow[dd, "\iota_G'", hook] \\
H \arrow[d, "h'"', hook] \arrow[r, "\iota_H"', hook]      & G +_{\Omega_n} H \arrow[rd, "!p", dashed]     &                           \\
H' \arrow[rr, "\iota_H'"', hook]                          &                                              & G' +_{\Omega_n} H'      
\end{tikzcd}
\end{center}
Notice that the arrows $\iota_G, \iota_G', \iota_H, \iota_H'$ are all monic. We have to establish that the morphism $p: G +_{\Omega_n}H \rightarrow G' +_{\Omega_n} H'$ (which is unique since it is a pushout arrow in $\mathbf{Poset}_{homo}$) is monic as well. Notice that $p$ can be defined piece-wise as the map taking any point $x$ in $G +_{\Omega_n}H$ to $(\iota'_G \circ g')(x)$ if $x$ is in $G$ and to $(\iota'_H \circ h')(x)$ otherwise. Since $G' +_{\Omega_n} H'$ is obtained by identifying the points in the image of $\Omega_n$ under $g'\circ g$ with the points in the image of $\Omega_n$ under $h'\circ h$, we have that, by its definition, $p$ must be injective and hence monic. 

Now we show that $(\mathbf{Poset}_{mono}, \Omega, \mathfrak{P})$ does not admit any S-functors. Assume for a contradiction that there exists an S-functor $F$ over $(\mathbf{Poset}_{mono}, \Omega, \mathfrak{P})$. Consider the linearly ordered posets $\Omega_3 = \{a \leq b \leq c\}$, $\Omega_2 = \{d \leq e\}$, and $\Omega_1 = \{x\}$. Since any spinal functor preserves the spine, we must have $F[\Omega_3] = 3$ and $F[\Omega_2] = 2$. Now consider the monomorphisms $f: \Omega_1 \rightarrow \Omega_3$ and $g: \Omega_1 \rightarrow \Omega_2$ given by $f(x) = c$ and $g(x) = d$. The pushout $P$ of $f,g$ is isomorphic to $\Omega_4$. Preservation of proxy pushouts immediately yields $4 = F[\Omega_4] = F[P] = \max\{2,3\} = 3$, a contradiction.
\end{proof}
\end{proposition}

Instead of exhaustively enumerating all possible obstructions to the existence of S-functors, we restrict our attention to those spined categories that come equipped with at least one S-functor. We shall see that the existence of a single S-functor already suffices to construct a functorial analogue of tree-width on any such category.

\begin{definition}\label{def:measurable-spined-category}
We call a spined category \emph{measurable} if it admits at least one S-functor.
\end{definition}

Of course $\nat$ is a measurable spined category. The measurability of $\GRPH_{mono}$ follows from Proposition~\ref{prop:S_functions_are_S_functors}, by noticing that that the clique number is an S-functor. However, this is a very special property enjoyed by $\GRPH_{homo}$.

\begin{proposition}\label{prop:generalized-clique-number-not-S-functor}
The generalized clique number $\omega$ need not give rise to an S-functor over an arbitrary measurable spined category.
\begin{proof}
Equip the natural numbers with the divisibility relation, and regard the resulting poset as a category $\mathbb{N}_{div}$. Equip $\mathbb{N}_{div}$ with the spine $$\Omega_n = \prod_{p \leq n} p^n$$ where $p$ ranges over the primes. The poset category $\mathbb{N}_{div}$ has all pushouts, the pushout of objects $n,m$ given by least common multiple of $n$ and $m$. Let $\mathfrak{P}(x \leq n, x \leq m)$ denote the least common multiple $\mathrm{lcm}(n,m)$. We verify each of the spined category properties in turn:
\begin{itemize}
    \item[\ref{property:SC1}:] Take any $n \in \mathbb{N}$. Let $p$ and $k$ denote respectively the largest prime and exponent which appears in the prime factorization of $n$. Then $n$ divides $\Omega_{p^k}$.
    \item[\ref{property:SC2}:] Immediate from Proposition~\ref{prop:pushouts-as-proxy-pushouts}.
\end{itemize}
Consider the map that sends each object $n \in \mathbb{N}_{div}$ to the highest exponent that occurs in the prime factorization of $n$ (OEIS~A051903~\cite{oeisA051903}). This is clearly an S-functor on the category $(\mathbb{N}_{div}, \Omega, \mathrm{lcm})$, which is therefore measurable. However, we claim that $\omega$ itself is not an S-functor on this spined category.

To see this, consider the objects $16$ and $81$ in $\mathbb{N}_{div}$. Since $\Omega_2 = 2^2 = 4$ and $\Omega_3 = 2^3\cdot 3^3 = 216$, the largest $n$ for which $\Omega_n$ divides $16$ is $\omega[16] = 2$. Similarly, $\omega[81] = 1$. However, we have $\omega[16 \cdot 81] = \omega[1296] = \omega[\Omega_4] = 4 \neq 2$.
\end{proof}
\end{proposition}

The reader may verify that, unlike the generalized clique number, the order map \textit{does} give rise to an S-functor over the category $\mathbb{N}_{div}$. However this is not true in general.

\begin{proposition}\label{prop:order-not-S-functor}
The order map $X \mapsto |X|$ need not give rise to an S-functor over an arbitrary measurable spined category.
\begin{proof}
The order map does not constitute an S-functor over the measurable spined category $\GRPH_{mono}$. Consider two copies of the graph with two vertices and one edge, glued together along a common vertex. If order was an S-functor, the resulting graph would have only two vertices.
\end{proof}
\end{proposition}

\section{Tree-width in a measurable spined category}\label{sec:triangulation_number_results}
 
In this section we give an abstract analogue of tree-width in our categorial setting, by proving a theorem in the style of Halin's Theorem \ref{thm:Halin_tw}. To do so, we must find a maximum S-functor (under the point-wise order). An obvious candidate is the map taking every object to its order (Definition \ref{def:order}). However, as we just saw (Proposition \ref{prop:order-not-S-functor}), the order need not constitute an S-functor for measurable spined categories. Thus, rather than trying to define an S-functor via morphisms from objects to elements of the spine, we will consider morphisms to elements of a distinguished class of objects which we call \emph{pseudo-chordal}. These objects will be used to define our abstract analogue of tree-width as an $S$-functor on any measurable spined category. We will conclude the section by showing how our abstract characterization of tree-width allows us to recover the familiar notions of graph and hypergraph tree-width.

\begin{definition}\label{def:pseudo-chordal-object}
We call an object $X$ of a spined category $(\mathcal{C}, \Omega, \mathfrak{P})$ \emph{pseudo-chordal} if for every two S-functors $F,G: \mathcal{C} \rightarrow \nat$ we have $F[X] = G[X]$ (if the spined category is not measurable, then every object is pseudo-chordal).
\end{definition}

\begin{proposition}\label{prop:pseudo-chordals-closed}
The set $Q$ of all pseudo-chordal objects of a spined category $(\mathcal{C}, \Omega, \mathfrak{P})$ contains all objects of the form $\Omega_n$, and is closed under proxy pushouts in the following sense: given two objects $A,B \in Q$ and two arrows $f: \Omega_n \rightarrow A$ and $g: \Omega_n \rightarrow B$, we always have $\mathfrak{P}(f,g) \in Q$.
\begin{proof}
Given two S-functors $F,G$ on $\mathcal{C}$, we always have $F[\Omega_n] = n = G[\Omega_n]$ by Property~\ref{property:SF1}. Moreover, by Property~\ref{property:SF2}, given $A,B \in Q$ and arrows $f: \Omega_n \rightarrow A$ and $g: \Omega_n \rightarrow B$, we have $F[\mathfrak{P}(f,g)] = \max\{F[A].F[B]\} = \max\{G[A].G[B]\} = G[\mathfrak{P}(f,g)]$.
\end{proof}
\end{proposition}

In light of Proposition \ref{prop:pseudo-chordals-closed}, it is natural to distinguish the smallest set of pseudo-chordal objects that contains the spine and which is closed under proxy pushouts. We call this set the set of \emph{chordal objects}. The name is given in analogy to chordal graphs: a resemblance that is best seen in the following recursive definition of chordal objects.

\begin{definition}\label{def:chordal-object}
We define the set of \textit{chordal objects} of the category spined category $\mathcal{C}$ inductively, as the smallest set $S$ of objects satisfying the following:
\begin{itemize}
    \item $\Omega_n \in S$ for all $n \in \mathbb{N}$, and
    \item $\mathfrak{P}(a,b) \in S$ for all objects $A,B \in S$ and arrows $a: \Omega_n \rightarrow A$ and $b: \Omega_n \rightarrow B$.
\end{itemize}
\end{definition}
 
Note that the notions of chordality and pseudo-chordality are well-defined even in \emph{non-measurable} categories (since every object is pseudo-chordal if the category in question is not measurable).

As an immediate consequence of Proposition \ref{prop:pseudo-chordals-closed} we have the following result.

\begin{corollary}\label{corollary:chordal_implies_pseudochordal}
All chordal objects are pseudo-chordal.
\end{corollary}

However, note that the converse of Corollary \ref{corollary:chordal_implies_pseudochordal} does not hold; as we shall see, it fails even in $\GRPH_{mono}$.

\begin{proposition}\label{prop:pseudochordal_not_imply_chordal}
Pseudo-chordality does not imply chordality.
\begin{proof}
We will show that, in the spined category $\GRPH_{mono}$, there exists a non-chordal object for which every pair of S-functors agree. To this end, consider, for some $n \in \mathbb{N}$, the element $K_n \#_{K_1} C_n$ obtained by identifying a vertex of an $n$-clique to a vertex of an $n$-cycle. Since $C_n$ is a subgraph of $K_n$, we have a sequence of injective graph homomorphisms $$K_n \hookrightarrow K_n \#_{K_1} C_n \hookrightarrow K_n \#_{K_1} K_n.$$
Thus, for any S-functor $F$,  we have \[n = F[K_n] \leq F[K_n \#_{K_1} C_n] \leq  F[K_n \#_{K_1} K_n] = \max\{ F[K_n], F[K_n] \} = n.\]
\end{proof}
\end{proposition}

We will use pseudo-chordal objects to define notion of a \emph{pseudo-chordal completion} of an object of a spined category. We point out that the name was given in analogy to the operation of a chordal completion of graphs (i.e. the addition of a set $F$ of edges to some graph $G$ such that the resulting graph $(V(G), E(G) \cup F)$ is chordal).

\begin{definition}
A \textit{pseudo-chordal completion} of an object $X$ of a spined category $(\mathcal{C}, \Omega, \mathfrak{P})$ is an arrow $\delta: X \hookrightarrow H$ for some pseudo-chordal object $H$. If the pseudo-chordal object $H$ is also chordal, then we call $\delta$ a \emph{chordal completion}.
\end{definition}

Note that, for graphs, one can give an alternative definition of the tree-width a graph $G$ as: $\tw(G) = \min \{\omega(H) -1 : H \text{ chordal completion of } G\}$ (where $\omega$ is the clique number) \cite{Diestel2010GraphTheory}. With this in mind, observe that the following definition of the \emph{width of a pseudo-chordal completion} furthers the analogy between our construction and the tree-width of graphs. 

\begin{definition}\label{def:width}
Let $X$ and $F$ be respectively an object and an S-functor in some measurable spined category. The \emph{width} of a pseudo-chordal completion $\delta: X \hookrightarrow H$ of $X$ is the value $F[H]$.
\end{definition}

We point out that, in contrast to the case of graphs, we do not define the width of a pseudo-chordal completion by using the generalized clique number $\omega$. This is because $\omega$ need not be an S-functor in general (by Proposition \ref{prop:generalized-clique-number-not-S-functor}). For clarity we note that the choice of S-functor in Definition \ref{def:width} is inconsequential since every two S-functors agree on pseudo-chordal objects (by the definition of pseudo-chordality). 

\begin{proposition}\label{prop:triangulation-is-s-functorial}
Let $(\mathcal{C}, \Omega, \mathfrak{P})$ be a measurable spined category and denote by $\Delta[X]$  and $\Delta^{ch}[X]$ the minimum possible width of respectively any pseudo-chordal completion of the object $X$ and any chordal completion of $X$. Then $\Delta$ and $\Delta^{ch}$ are functors from $\mathcal{C}$ to $\mathbb{N}_{\leq}$.
\begin{proof}
We only prove the claim for $\Delta$ since the argument for $\Delta^{ch}$ is the same. Let $F$ be any S-functor over $(\mathcal{C}, \Omega, \mathfrak{P})$. We need to verify that, for every arrow $f: X \to Y$ in $\mathcal{C}$, we have $\Delta[X] \leq \Delta[Y]$. To this end take any such arrow $f: X \to Y$ and two minimum-width pseudo-chordal completions $\delta_X :X \to H_X$ and $\delta_Y: Y \to H_Y$ of $X$ and $Y$ respectively. Since $\delta_Y \circ f$ is also a pseudo-chordal completion of $X$ and by the minimality of the width of $\delta$, we have $\Delta[X] = F[H_X] \leq F[H_Y] = \Delta[Y]$.
\end{proof}
\end{proposition}

\begin{definition}\label{def:triangulation-functor}
Let $\Delta$ and $\Delta^{ch}$ be the functors defined in Proposition \ref{prop:triangulation-is-s-functorial}. We call $\Delta$ the \textit{triangulation functor} and $\Delta^{ch}$ the \emph{chordal triangulation functor}. 
\end{definition}

Our goal now is to show that the triangulation functor of a measurable spined category is in fact an S-functor. Specifically we prove our main theorem which states that both $\Delta$ and $\Delta^{ch}$ are S-functors in any measurable spined category. 

\begin{theorem}\label{thm:triang-functor-is-S-functor}
Both the triangulation and chordal-triangulation functors are S-functors in any measurable spined category.
\begin{proof}
Let $(\mathcal{C}, \Omega, \mathfrak{P})$ be any measurable spined category equipped with some S-functor $F$. We will prove the statement only for $\Delta$ since the method of proof for the $\Delta^{ch}$ case is the same.

Consider a pseudo-chordal completion $c: X \to H$ of a pseudo-chordal object~$X$. Then $F[X] \leq F[H]$, and so the identity pseudo-chordal completion of $X$ has smaller width than any other pseudo-chordal completion of $X$. This proves that $\Delta[\Omega_n] = n$ and hence that $\Delta$ satisfies property \ref{property:SC1}. 

For \ref{property:SC2}, consider any span 
\begin{tikzcd}
A & \Omega_n \arrow[l, "a"'] \arrow[r, "b"] & B
\end{tikzcd} in $\mathcal{C}$. We have to prove that $\Delta[\mathfrak{P}(a,b)] = \max \{\Delta[A], \Delta[B]\}$. Choose a pseudo-chordal completion $\alpha: A \to H_A$ (resp. $\beta: B \to H_B$) for which $F[H_A]$ (resp. $F[H_B]$) is minimal. Using property \ref{property:SC2}, there is a unique arrow $(\alpha, \beta): \mathfrak{P}(a, b) \to \mathfrak{P}(\alpha a, \beta b)$ such that the following diagram commutes.
\begin{center}
\begin{tikzcd}
\Omega_n \arrow[r, "a"] \arrow[d, "b"'] & A \arrow[r, "\alpha"] \arrow[d]                            & H_A \arrow[dd]                    \\
B \arrow[d, "\beta"'] \arrow[r]         & {\mathfrak{P}(a,b)} \arrow[rd, "{(\alpha,\beta)}", dashed] &                                   \\
H_B \arrow[rr]                          &                                                            & {\mathfrak{P}(\alpha a, \beta b)}
\end{tikzcd}
\end{center}
Now take a pseudo-chordal completion $\delta: \mathfrak{P}(a, b) \to H$ of $\mathfrak{P}(a, b)$ for which the quantity $F[H]$ is minimal. Consider the following diagram.
\begin{center}
\begin{tikzcd}
\Omega_n \arrow[rr, "a"] \arrow[rd, "b"] \arrow[dd, "\mathbf{id}_{\Omega_n}", no head] &                                                               & A \arrow[rd, "{\mathfrak{P}(a,b)_{a}}" description] \arrow[dd, "\alpha", near start]                                          &                                                                                   &                                                                              \\
                                                                                       & B \arrow[rr, "{\mathfrak{P}(a,b)_{b}}", near start, crossing over]                              &                                                                            & {\mathfrak{P}(a,b)} \arrow[rd, "\delta"] \arrow[dd, "{(\alpha, \beta)}"', dashed] &                                                                              \\
\Omega_n \arrow[rr, "\alpha a"', near end] \arrow[rd, "\beta b"']                                &                                                               & H_A \arrow[rd, "{\mathfrak{P}(\alpha a, \beta b)_{\alpha a}}" description] &                                                                                   & H                                                                            \\
                                                                                       & H_B \arrow[from=uu, "\beta", near start, crossing over] \arrow[rr, "{\mathfrak{P}(\alpha a, \beta b)_{\beta b}}", near start] &                                                                            & {\mathfrak{P}(\alpha a, \beta b)}                                                &                                                                              \\
n \arrow[rd, "F(\beta b)"'] \arrow[rr, "F(\alpha a)"]                                  &                                                               & {F[H_A] = \Delta[A]} \arrow[rd, "F{\mathfrak{P}(\alpha a, \beta b)_{\alpha a}}" description]                                      &                                                                                   & {F[H] = \Delta[\mathfrak{P}(a,b)]} \arrow[ld, "\mu_2"', dotted, bend right] \\
                                                                                       & {F[H_B] = \Delta[B]} \arrow[rr, "F{\mathfrak{P}(\alpha a, \beta b)_{\beta b}}"]                         &                                                                            & {F[\mathfrak{P}(\alpha a, \beta b)]} \arrow[ru, "\mu_1"', dotted, bend right]     &                                                                             
\end{tikzcd}
\end{center}
To show that $\Delta[\mathfrak{P}(a,b)] = \max \{\Delta[A], \Delta[B]\}$, it suffices to deduce the existence of the dotted arrows $\mu_1$ and $\mu_2$ in the diagram above. 

Note that, since $F$ is an S-functor, the bottom square (which is a diagram in $\nat$) commutes and $\mathfrak{P}(\alpha a, \beta b) = \max \{F[H_A], F[H_B]\}$. Since $\delta \circ \mathfrak{P}(a,b)_{a}$ constitutes a pseudo-chordal completion of $A$ and since we chose $H$ so that $F[H]$ is minimal, we have $F[H_A] \leq F[H]$. Similarly we can deduce $F[H_B] \leq F[H]$. Thus we have
$$ F[\mathfrak{P}(\alpha a, \beta b)] = \max \{F[H_A], F[H_B]\} \leq F[H], $$
which proves the existence of $\mu_1$. 

By Proposition \ref{prop:pseudo-chordals-closed}, we know that the set of pseudo-chordal objects is closed under proxy pushouts. Since $H_A$ and $H_B$ are pseudo-chordal, so is their proxy pushout $\mathfrak{P}(\alpha a, \beta b)$. Hence $(\alpha, \beta): \mathfrak{P}(a, b) \to \mathfrak{P}(\alpha a, \beta b)$ is a pseudo-chordal completion of $\mathfrak{P}(a, b)$. However, so is $H$. In fact we chose $H$ so that $F[H]$ is minimal (since $F[H] = \Delta[\mathfrak{P}(a,b)]$). Thus we have $F[\mathfrak{P}(\alpha a, \beta b)] \geq F[H]$, which proves the existence of $\mu_2$.
\end{proof}
\end{theorem}

Surprisingly, the S-functors $\Delta$ and $\Delta^{ch}$ defined above coincide. 

\begin{corollary}\label{corollary:Delta_and_Delta_chord_are_same}
In any measurable spined category we have $\Delta = \Delta^{ch}$.
\begin{proof}
Consider any measurable spined category $(\mathcal{C}, \Omega, \mathfrak{P})$ equipped with an S-functor $F$ and let $X$ be an object in $\mathcal{C}$. Since every chordal object is also pseudo-chordal (Corollary \ref{corollary:chordal_implies_pseudochordal}) we know that $\Delta[X] \leq \Delta^{ch}[X]$. We now show that given any minimum-width pseudo-chordal completion $\delta: X \rightarrow H$ of $X$, we can find a chordal completion of $X$ of the same width as $\delta$. 

Let $\gamma: H \to H^{ch}$ be a minimum-width chordal completion of $H$. Since $H$ is pseudo-chordal, all S-functors take the same value on $H$. In particular this means that $\Delta[H] = \Delta^{ch}[H]$ since both $\Delta$ and $\Delta^{ch}$ are S-functors by Theorem \ref{thm:triang-functor-is-S-functor}. Thus we have $F[H] = \Delta[H] = \Delta^{ch}[H] = F[H^{ch}]$. But then $\gamma \circ \delta$ is a chordal completion of $X$ with width $F[H^{ch}] = F[H]$, as desired.
\end{proof}
\end{corollary}

The S-functor $\Delta$ constructed above satisfies a maximality property broadly analogous to  Theorem~\ref{thm:Halin_tw}.

\begin{theorem}\label{thm:S_funtor_lattice}
Let $(\mathcal{C}, \Omega, \mathfrak{P})$ be any measurable spined category. The set of all S-functors over $(\mathcal{C}, \Omega, \mathfrak{P})$ is a join semi-lattice under the pointwise ordering with $\Delta$ as its maximum element.
\end{theorem}
\begin{proof}
Let $\mathcal{Z}$ be any non-empty (possibly infinite) subset of the set of S-functors over $(\mathcal{C}, \Omega, \mathfrak{P})$. In what follows we shall first construct the supremum of $\mathcal{Z}$ and then we will prove that it constitutes an S-functor.

Define the map $F_{\mathcal{Z}} : \mathcal{C} \to \mathbb{N}$ for any $W$ in $\mathcal{C}$ as $F_{\mathcal{Z}}[W] := \max_{F' \in \mathcal{Z}} F'[W]$. (Note that this maximum always exists since every object $X$ is mapped by any S-functor to at most the value of $|X|$ and hence $\{F'[X] : F' \in \mathcal{Z}\}$ is a bounded set of integers.)

We claim that, for any arrow $m: X \to Y$ in $\mathcal{C}$, we have $F_{\mathcal{Z}}[X] \leq F_{\mathcal{Z}}[Y]$. To see this, let $Q$ be an element of $\mathcal{Z}$ such that $Q[X] = F_\mathcal{{Z}}[X]$ (by the definition of $F_{\mathcal{Z}}$ and since $\mathcal{Z}$ is non-empty, such a $Q$ always exists). The functoriality of $Q$ implies that, if there is an arrow $X \to Y$ in $\mathcal{C}$, then $Q[X] \leq Q[Y]$; in particular we can deduce that
\[F_{\mathcal{Z}}[X] = Q[X] \leq Q[Y] \leq \max_{F' \in \mathcal{Z}} F'[Y] = F_{\mathcal{Z}}[Y].\]
Hence there is an arrow $g: F_{\mathcal{Z}}[X] \to F_{\mathcal{Z}}[Y]$ in $\nat$, which means that we can (slightly abusing notation) render $F_{\mathcal{Z}}$ a functor by extending the definition of $F_{\mathcal{Z}}$ to map any arrow $m: X \to Y$ to the arrow $g: F_{\mathcal{Z}}[X] \to F_{\mathcal{Z}}[Y]$ in $\nat$. 

From what we showed above, we know that $F_{\mathcal{Z}}$ is a functor. Now we will show that it is spinal functor. Note that $F_{\mathcal{Z}}$ clearly preserves the spine; furthermore, for any span \begin{tikzcd} A &\Omega_n \arrow[l, "a"'] \arrow[r, "b"] & B\end{tikzcd}, we have 
\begin{align*}
F_{\mathcal{Z}}[\mathfrak{P}(a,b)] &= \max_{F' \in \mathcal{Z}} F'[\mathfrak{P}(a,b)]  &\text{(by the definition of } F_{\mathcal{Z}}) \\ &= \max_{F' \in \mathcal{Z}} \max \{F'[A], F'[B]\} &\text{(since } F' \text{ is an S-functor)} \\
&= \max \{ F_{\mathcal{Z}}[A], F_{\mathcal{Z}}[B]\}.
\end{align*}
Thus $F_{\mathcal{Z}}$ is an S-functor since it satisfies Properties \ref{property:SF1} and \ref{property:SF2}.  In particular we have proved that the set of all S-functors over $(\mathcal{C}, \Omega, \mathfrak{P})$ is a join semi-lattice under the point-wise ordering. 

To see that $\Delta$ is the largest element of this semi-lattice, take any pseudo-chordal completion $\delta: X \to H$ of some object $X$. For any $S$-functor $F$, the following diagram commutes (by functoriality). 
\begin{center}\begin{tikzcd}
        X \arrow[d, "F"'] \arrow[r, "\delta"] & H \arrow[d, "F"] \\
        F[X] \arrow[r, "F_\delta"']       & F[H]      \end{tikzcd}\end{center}
But since $\Delta[X] := F[H]$, we have $F[X] \leq \Delta[X]$ and hence $\Delta$ is the maximum element of the join semi-lattice of S-functors.
\end{proof}

\section{Abstract analogues of tree-width.}\label{sec:examples}
In this section we will find instantiations of spined categories such that their triangulation numbers recover tree-width, hypergraph tree-width and complemented tree-width (i.e. the invariant $G \mapsto \tw(\overline{G}) + 1$) . 

\subsection{Tree-width of graphs and hypergraphs}
Earlier we showed (Proposition \ref{prop:S_functions_are_S_functors}) that every S-function yields an S-functor over $\GRPH_{mono}$. The next result goes further than this and shows that the triangulation functor on $\GRPH_{mono}$ takes every graph $G$ to $\tw(G) + 1$.

\begin{corollary}\label{corollary:spinal-tree-width-homo}
Let $\Delta$ be the triangulation functor of $\GRPH_{mono}$. Then, for any graph $G$, we have $\Delta[G] = \tw(G) + 1$.
\begin{proof}
In $\GRPH_{mono}$ the generalized clique-number agrees with the clique number. Hence we compute 
\begin{align*}
    \tw(G) + 1 &= \min \{\omega(H) : H \text{ is a chordal completion of } G\} \text{ (see \cite{Diestel2010GraphTheory})}\\
    &= \Delta^{ch}[G] \text{ (since } \omega 
    \text{ is an S-functor in } \GRPH_{mono}) \\
    &= \Delta[G] \text{ (by Corollary \ref{corollary:Delta_and_Delta_chord_are_same})}.
\end{align*}
\end{proof}
\end{corollary}

Next we consider the category $\HGr_{mono}$ of hypergraphs and their injective homomoprhisms which we describe now. Let $H_1$ and $H_2$ be hypergraphs; a vertex map $h: V(H_1) \to V(H_2)$ is a \emph{hypergraph homomorphism} if it preserves hyper-edges; that is to say that, for every edge $F \in E(H_1)$, the set $h(F) := \{h(x) : x \in F\}$ is a hyper-edge in $H_2$. Hypergraph homomorphisms clearly compose associatively, thus we can define the category $\HGr_{mono}$ which has finite hypergraphs as objects and injective hypergraph homomorphisms as arrows.

\begin{theorem}\label{thm:hypergraphs}
Let $\Omega : \mathbb{N}_{=} \to \mathcal \HGr$ be the functor taking every integer $n$ to the hypergraph $([n], 2^{[n]})$ and let $\mathfrak{P}$ assign to each span of the form
\begin{tikzcd}
H_1 & \Omega_n \arrow[l, "h_1"'] \arrow[r, "h_2"] & H_2
\end{tikzcd} in $\HGr$ the cocone \begin{tikzcd}
H_1 \arrow[rr, "{\mathfrak{P}(h_1,h_2)_{h_1}}"']&& \mathfrak{P}(h_1, h_2) && H_2 \arrow[ll, "{\mathfrak{P}(h_1,h_2)_{h_2}}"']
\end{tikzcd} where
\[\mathfrak{P}(h_1, h_2) := \Bigl( \bigl( V(H_1) \uplus V(H_2) \bigr )/_{h_1 = h_2}, \; \: \bigl( E(H_1) \uplus E(H_2) \bigr )/_{h_1 = h_2} \Bigr)\]
and $\mathfrak{P}(h_1,h_2)_{h_i}$ is the map taking every vertex $v$ in $H_i$ to $v_i$ in $\mathfrak{P}(h_1,h_2)$. 
Then the triple $(\HGr, \Omega, \mathfrak{P})$ is a spined category.
\begin{proof}
Clearly Property \ref{property:SC1} is satisfied, so, to show Property \ref{property:SC2}, consider the following diagram in $\HGr$ (we will argue for the existence and uniqueness of $(j_1, j_2)$.
\begin{center}
\begin{tikzcd}
\Omega_n \arrow[rr, "h_1"] \arrow[d, "h_2"]                         &  & H_1 \arrow[r, "j_1"] \arrow[d, "{{\mathfrak{P}(h_1,h_2)_{h_1}}}"] & J_1 \arrow[dd]                  \\
H_2 \arrow[d, "j_2"] \arrow[rr, "{{\mathfrak{P}(h_1, h_2)_{h_2}}}"] &  & {\mathfrak{P}(h_1,h_2)} \arrow[rd, "{(j_1,j_2)}", dashed]         &                                 \\
J_2 \arrow[rrr]                                                     &  &                                                                   & {\mathfrak{P}(j_1 h_1,j_2 h_2)}
\end{tikzcd}
\end{center}
We define $(j_1,j_2): \mathfrak{P}(h_1, h_2) \to \mathfrak{P}(j_1, j_2)$ as 
\[(j_1,j_2)(x) := 
    \begin{cases}
    j_1(x) \text{ if } x \in V(H_1) \cap \mathfrak{P}(h_1, h_2) \\
    j_2(x) \text{ otherwise}.
    \end{cases}
\]
Clearly $(j_1,j_2)$ is the unique injective vertex-map making the diagram commute (this can be easily seen by considering the forgetful functor taking every hypergraph to its vertex-set). Furthermore, by recalling the definition of the proxy pushout, one can easily see that it is in fact an injective hypergraph homomorphism, as desired.
\end{proof}
\end{theorem}

Note that we can also construct a spined functor from the spined category $\HGr$ of hypergraphs to the spined category $\GRPH_{mono}$ of graphs. We do this by observing that the mapping $\mathfrak{G} : \HGr \to \GRPH_{mono}$ which associates every hypergraph to its Gaifman graph (sometimes also referred to as `primal graph') is clearly functorial.

\begin{proposition}\label{prop:Gaifman}
The Gaifman graph functor $\mathfrak{G}: \HGr \to \GRPH_{mono}$ is a spined functor.
\begin{proof}
Note that $\mathfrak{G}[([n], 2^{[n]})] = K_n$ (i.e. $\mathfrak{G}$ satisfies Property \ref{property:SF1}). Now take the proxy pushout $\mathfrak{P}(h_1,h_2)$ of some span \begin{tikzcd}
H_1 & \Omega_n \arrow[l, "h_1"'] \arrow[r, "h_2"] & H_2 \end{tikzcd} in $\HGr$. Recall that $\mathfrak{P}(h_1,h_2)$ is constructed by identifying $H_1$ and $H_2$ along $\Omega_n := ([n], 2^{[n]})$. Thus, since $\mathfrak{G}$ preserves the spine (as we just showed) we know that the Gaifman graph $\mathfrak{G}[\mathfrak{P}(h_1,h_2)]$ of $\mathfrak{P}(h_1,h_2)$ is given by the clique-sum along a $K_n$ of the Gaifman graphs of $H_1$ and $H_2$. In other words we have $\mathfrak{G}[\mathfrak{P}(h_1,h_2)] = \mathfrak{G}[H_1] \#_{\mathfrak{G}[\Omega_n]} \mathfrak{G}[H_1]$ which proves that $\mathfrak{G}$ satisfies Property \ref{property:SF2}. Thus $\mathfrak{G}$ is a spined functor.
\end{proof}
\end{proposition}

\begin{corollary}
The spined category $\HGr$ is measurable; in particular there are uncountably many S-functors over $\HGr$.
\begin{proof}
Immediate from Propositions~\ref{prop:S_functions_are_S_functors}~and~\ref{prop:Gaifman}.
\end{proof}
\end{corollary}

\noindent Now consider any proxy pushout $\mathfrak{P}(h_1,h_2)$ of a span \begin{tikzcd}
H_1 & \Omega_n \arrow[l, "h_1"'] \arrow[r, "h_2"] & H_2
\end{tikzcd} in $\HGr$. It follows (in much the same way as it does for graphs) that the tree-width of $\mathfrak{P}(h_1,h_2)$ is the maximum of $\tw(H_1)$ and $\tw(H_2)$. Since, by the definition of tree-width, we have $\tw(([n], 2^{[n]})) = n - 1$, it follows that, in $(\HGr, \Phi)$, $\Delta(K) = \tw(K) + 1$ for any chordal object $K$ in $(\HGr, \Phi)$. Thus we shave shown the following result. 

\begin{corollary}
If $\Delta$ is the triangulation number of $(\HGr, \Omega, \mathfrak{P})$, then, for any hypergraph $H$, $\Delta(H) = \tw(H) + 1$.
\end{corollary}

\subsection{Complemented tree-width}\label{sec:reflexive homomorphisms}
Another example of a spined category is given by taking the discrete graphs $(\overline{K_n})_{n \in \mathbb{N}}$ as the spine for the category $\Rmono$ of graphs and \emph{reflexive monomorphisms} (which we define in what follows). 

\begin{definition}\label{def:R-homo}
A vertex map $f: V(G) \to V(H)$ is a \emph{reflexive homomorphism} from the graph $G$ to the graph $H$ if the following implication holds for all pairs of vertices $x$ and $y$ in $G$: $f(x)f(y) \in E(H) \Rightarrow xy \in E(G)$.
\end{definition}

In what follows we denote by $\Rhomo$ the category having graphs as objects and reflexive homomorphisms as arrows, while we denote by $\Rmono$ the category of graphs and injective reflexive homomorphisms.

\begin{proposition}\label{prop:R-homo-indep-set}
If $f: \overline{K}_n \to H$ is an arrow in $\Rmono$, then the image of $f$ in $H$ is an independent set.
\end{proposition}
\begin{proof}
B.w.o.c. if $f(x)f(y) \in E(H)$, then $xy \in E(\overline{K_n})$ even though $E(\overline{K_n}) = \emptyset$.
\end{proof}

Since $\overline{K_n}$ has no edges, every graph on at-most $n$ vertices has an injective reflexive homomorphism to $\overline{K_n}$. Thus $(\overline{K_n})_{n \in \mathbb{N}}$ satisfies Property~\ref{property:SC1} in $\Rmono$ and hence forms a suitable spine. 

Now, we will show that there is an appropriate choice of a proxy-pushout operation -- which we will denote as $\mathfrak{I}$ -- which turns $(\Rmono, (\overline{K_n})_{n \in \mathbb{N}}, \mathfrak{I})$ into a measurable spined category.

A first, but naive candidate for $\mathfrak{I}$ is the operation taking any span of the form \begin{tikzcd}
L & \overline{K_n} \arrow[l, "\ell"'] \arrow[r, "r"] & R
\end{tikzcd} in $\Rmono$ to the graph obtaned by 'gluing' $L$ to $R$ along their shared independent set (c.f. Proposition~\ref{prop:R-homo-indep-set}) of size $n$ (i.e. identify the image of $\overline{K_n}$ in $L$ to the image of $\overline{K_n}$ in $R$). Notice, though, that if we took $L \cong \overline{K}_\ell$ and $R \cong \overline{K}_r$, then this construction would produce as their proxy-pushout the graph $\overline{K}_{\ell + r - n}$. However, this would then preclude the existence of any S-functor $F$ over this spined category since such an $F$ would have to simlutaneously satisfy $F[\overline{K}_{\ell + r - n}] = \ell + r - n$ (because $\overline{K}_{\ell + r - n}$ is in the spine) and also $F[\overline{K}_{\ell + r - n}] = \max \{\ell, r\}$ (because $\max$ is the proxy pushout of $\nat$). Thus, with these considerations in mind, we come to the following definition of $\mathfrak{I}$.

\begin{definition}
Let $\mathfrak{I}$ be the operation taking every span in $\Rmono$ of the form
\begin{tikzcd}
    L & \overline{K}_n \arrow[l, "\ell"'] \arrow[r, "r"] & R
\end{tikzcd} 
 to the cospan 
\begin{tikzcd}
    L \arrow[r, "{\mathfrak{I}(\ell,r)_\ell}"] & {\mathfrak{I}(\ell,r)} & R \arrow[l, "{\mathfrak{I}(\ell,r)_r}"']
\end{tikzcd}
which we define as follows. The graph $\mathfrak{I}(\ell,r)$ is given by identifying $L$ and $R$ along their shared $n$-vertex independent set and then adding edges to this resulting graph so as to make $L\setminus \ell(\overline{K_n})$ complete to $R\setminus \ell(\overline{K_n})$; in other words $\mathfrak{I}(\ell,r)$ is the graph with vertex-set $(V(L)\uplus V(R))/_{\ell = r}$ and edge-set
\[E(L) \uplus E(R) \uplus \Bigl(  \bigl(V(L)\setminus \ell(\overline{K_n})\bigr) \times \bigl(V(R)\setminus r(\overline{K_n})\bigr)\Bigr).\]
Finally, the arrows $\mathfrak{I}(\ell,r)_\ell$ and $\mathfrak{I}(\ell,r)_r$ are just the obvious injections taking $L$ and $R$ respectively into $\mathfrak{I}(\ell,r)$ (these are easily seen to be reflexive homomorphisms).
\end{definition}

\begin{proposition}\label{prop:IS-category-spined}
The triple $(\Rmono, (\overline{K_n})_{n \in \mathbb{N}}, \mathfrak{I})$ is a spined category. 
\end{proposition}
\begin{proof}
As we already observed, Property~\ref{property:SC1} is satisfied. For Property~\ref{property:SC2}, we are given any diagram
\begin{tikzcd}
L' & L \arrow[l, "\ell'"'] & \overline{K_n} \arrow[l, "\ell"'] \arrow[r, "r"] & R \arrow[r, "r'"] & R'
\end{tikzcd}
in $\Rmono$, and we need to demonstrate the existence of a unique arrow $m: \mathfrak{I}(\ell, r) \to \mathfrak{I}(\ell' \ell, r' r)$ which makes the following diagram commute. 
\begin{center}
\begin{tikzcd}
\overline{K_n} \arrow[d, "\ell"'] \arrow[r, "r"]                 & R \arrow[r, "r'"] \arrow[d, "{\mathfrak{I}(\ell, r)_r}"] & R' \arrow[dd, "{\mathfrak{I}(\ell' \ell, r' r)_{r'r}}"] \\
L \arrow[d, "\ell'"'] \arrow[r, "{\mathfrak{I}(\ell, r)_\ell}"'] & {\mathfrak{I}(\ell, r)} \arrow[rd, "m", dotted]          &                                                         \\
L' \arrow[rr, "{\mathfrak{I}(\ell' \ell, r' r)_{\ell'\ell}}"]    &                                                          & {\mathfrak{I}(\ell' \ell, r' r)}                       
\end{tikzcd}
\end{center}
We use $\ell'$ and $r'$ to define $m$ piece-wise as follows: 
\[
m: x \mapsto 
\begin{cases}
\bigl(\mathfrak{I}(\ell' \ell, r' r)_{r'r} \circ r'\bigr) (x) \text{ if } x \in V(R) \cap V(\mathfrak{I}(\ell, r)) \\
\bigl(\mathfrak{I}(\ell' \ell, r' r)_{\ell' \ell} \circ \ell'\bigr) (x) \text{ otherwise}.
\end{cases}
\]
Clearly $m$ is the unique injective \emph{vertex-map} that makes the above diagram commute; thus all that remains to be shown is that $m$ is indeed a reflexive homomorphism. However, this follows immediately from the definition of $\mathfrak{I}$ and from the fact that $r'$ and $\ell'$ are $R$-homomorphisms.
\end{proof}

\begin{proposition}\label{prop:complement-iso-of-cats}
The complementation map $\overline{(-)}: \Rmono \to \GRPH_{mono}$ is a functor and indeed it is an isomorphism of categories.
\end{proposition}
\begin{proof}
First notice that $\overline{(-)}$ preserves identity arrows and it is a bijection on objects; so now consider an arrow $a: A \to B$ in $\Rmono$ we claim that the vertex map specified by $a$ constitutes an arrow from $\overline{A}$ to $\overline{B}$ in $\GRPH_{mono}$. To see this, take any edge $xy \in E(\overline{A})$; since $x$ and $y$ are not adjacent in $A$, then $a(x)a(y) \not \in E(B)$ (since $a$ is a reflexive homomorphism) and thus $a(x)a(y) \in E(\overline{B})$.

Conversely, if  $a: \overline{A} \to \overline{B}$ is an arrow in $\GRPH_{mono}$, for each pair $a(x)a(y) \in E(B)$ we must have $xy \not \in E(\overline{A})$ (since $a$ is a graph monomorphism and thus $A$ is a subgraph of $B$) which is equivalent to saying that $a(x)a(y) \in E(B)$ implies $xy \in E(A)$, as desired. Thus we have that $\overline{(-)}$ is a functor which is bijective on objects, bijective on arrows, full and faithful. Furthermore, it is easily seen that complementation is self-inverse, so it is the desired isomorphism of categories. 
\end{proof}

\begin{corollary}
The isomorphism $\overline{(-)}: \Rmono \to \GRPH_{mono}$ is a spined functor \[\overline{(-)}: (\Rmono, (\overline{K_n})_{n \in \mathbb{N}}, \mathfrak{I}) \to (\GRPH_{mono}, (K_n)_{n \in \mathbb{N}}, \#).\]
\end{corollary}
\begin{proof}
By Proposition~\ref{prop:complement-iso-of-cats}, $\overline{(-)}$ is a functor and indeed it is an isomorphism of categories. It clearly preserves the spine and it also takes any proxy-pushout square\[
\begin{tikzcd}
\overline{K_n} \arrow[d, "\ell"'] \arrow[r, "r"] & R \arrow[d]              \\
L \arrow[r]                                      & {\mathfrak{I}(\ell, r))}
\end{tikzcd} \text{ to the proxy-pushout square }
\begin{tikzcd}
K_n \arrow[r, "\overline{r}"] \arrow[d, "\overline{\ell}"] & \overline{R} \arrow[d]             \\
\overline{L} \arrow[r]                                     & \overline{L} \#_{K_n} \overline{R},
\end{tikzcd}. \] Thus $\overline{(-)}$ is a spined functor.
\end{proof}

Since spined functors compose and since $(\GRPH_{mono}, (K_n)_{n \in \mathbb{N}}, \#)$ is measurable, we also immediately have the following result. 

\begin{corollary}
The spined category $(\Rmono, (\overline{K_n})_{n \in \mathbb{N}}, \mathfrak{I})$ is measurable.
\end{corollary}

An example of an S-functor over $(\Rmono, (\overline{K_n})_{n \in \mathbb{N}}, \mathfrak{I})$ is its generalized clique number $\dot{\omega}_{R}$ (i.e. the function associating to each graph its independence number). This can be easily seen either by factoring it through the complementation map as $\dot{\omega}_{R} = \omega \circ \overline{(-)}$ or by simply checking from first principles that $\dot{\omega}_{R}$ satisfies Properties~\ref{property:SF1} and~\ref{property:SF2}. 

\begin{corollary}
Letting $\Delta$ be the triangulation functor of $(\Rmono, (\overline{K_n})_{n \in \mathbb{N}}, \mathfrak{I})$, we have $\Delta[G] = \tw(\: \overline{G}\:) + 1$ for all graphs $G$.
\end{corollary}

\section{New Spined Categories from Old}\label{sec:new_from_old}
The spined categories encountered so far came equipped with their ``standard'' notion of (mono)morphism: posets with monotone maps, graphs with graph homomorphsims, hypergraphs with hypergraph homomorphisms or graphs with reflexive homomorphisms. In particular, for a class $S$ of combinatorial objects decorated with extraneous structure (such as colored or labeled graphs), the appropriate choice of morphism may be less obvious. In these cases, a ``forgetful'' function $f: S \rightarrow \mathcal{C}$ from $S$ to some spined category $\mathcal{C}$ allows us to study properties of $S$ by studying properties of its image in $\mathcal{C}$.

It is straightforward to check that we can define a category $S_{\downarrow f}$, which we call the \emph{$S$-category induced by $f$} by taking $S$ itself as the collection of objects of $S_{\downarrow f}$ and, for any two objects $A$ and $B$ in $S$, setting $\homset_{S_{\downarrow f}}(A, B) := \homset_{\mathcal{C}}(f(A), f(B))$. 

It will be convenient to notice that -- up to categorial isomorphism -- $f^{-1}(X)$ (for any object $X$ in the range of $f$) consists of only one object in $S_{\downarrow f}$. To see this, suppose $f$ is not injective (otherwise there is nothing to show) and let $A, B \in S$ be elements of the set $f^{-1}(X)$. By the definition of $S_{\downarrow f}$, we know that $\id_{X} \in \homset_{S_{\downarrow f}}(A,B)$ since $\homset_{S_{\downarrow f}}(A,B) = \homset_{\mathcal{C}}(A,B)$. Thus $A$ and $B$ are isomorphic in $S_{\downarrow f}$ since identity arrows are always isomorphisms.

Note that by the construction of $S_{\downarrow f}$, the function $f$ actually constitutes a faithful and injective (on objects and arrows) functor from $S_{\downarrow f}$ to $\mathcal{C}$. The next result shows that if $\mathcal{C}$ is spined and if the range of $f$ is sufficiently large, then we can chose a spine $\Omega^S$ and a proxy pushout $\mathfrak{P}^S$ on $S_{\downarrow f}$ which turn $(S_{\downarrow f}, \Omega^S, \mathfrak{P}^S)$ into a spined category and $f: (S_{\downarrow f}, \Omega^S, \mathfrak{P}^S) \rightarrow (\mathcal{C}, \Omega, \mathfrak{P})$ into a spined functor.

\begin{theorem}\label{thm:new_spined_categories_from_old}
Let $(\mathcal{C}, \Omega, \mathfrak{P})$ be a spined category, $S$ be a set and $f: S \to \mathcal{C}$ be a function. If $f$ is both
\begin{enumerate}
    \item surjective on the spine of $\mathcal{C}$ (i.e. $\forall n \in \mathbb{N}, \exists X \in S$ s.t. $f(X) = \Omega_n$) and such that
    \item for every span \begin{tikzcd}
f(X) & \Omega_n \arrow[l, "x"'] \arrow[r, "y"] & f(Y)
\end{tikzcd} in $\mathcal{C}$, there exists a distinguished element $Z_{x,y} \in S$ such that $f(Z_{x,y}) = \mathfrak{P}(x,y)$,
\end{enumerate}
then we can choose a functor $\Omega^S$ and operation $\mathfrak{P}^S$ such that
\begin{itemize}
    \item $(S_{\downarrow f}, \Omega^S, \mathfrak{P}^S)$ is a spined category and 
    \item $f$ is a spinal functor from $(S_{\downarrow f}, \Omega^S, \mathfrak{P}^S)$ to $(\mathcal{C}, \Omega, \mathfrak{P})$ 
    \item if $(\mathcal{C}, \Omega, \mathfrak{P})$ is a measurable spined category, then so is $(S_{\downarrow f}, \Omega^S, \mathfrak{P}^S)$.
\end{itemize} 
\begin{proof}
Define $\Omega^S$ and $\mathfrak{P}^S$ as follows: \begin{itemize}
    \item $\Omega^S: \mathbb{N}_{=} \to S_{\downarrow f}$ is the functor taking each $n$ to an element of $f^{-1}(\Omega_n)$ (we can think of this as picking a representative of the equivalence class $f^{-1}(\Omega_n)$ for each $n$ since, as we observed earlier, all elements of $f^{-1}(\Omega_n)$ are isomorphic),
    \item $\mathfrak{P}^S$ is the operation assigning to each span \begin{tikzcd}
X & \Omega_n \arrow[l, "x"'] \arrow[r, "y"] & Y
\end{tikzcd} in $S_{\downarrow f}$ the cocone \begin{tikzcd}
X \arrow[r, "{\mathfrak{P}(x,y)_x}"] & {\mathfrak{P}^S(x,y) := Z_{x,y}} & Y \arrow[l, "{\mathfrak{P}(g,h)_h}"']
\end{tikzcd}, where $Z_{x,y}$ is the distinguished element whose existence is guaranteed by the second property of $f$.
\end{itemize}

Now we will show that $(S_{\downarrow f}, \Omega^S, \mathfrak{P}^S)$ is a spined category. Property \ref{property:SC1} holds in $(S_{\downarrow f}, \Omega^S, \mathfrak{P}^S)$ since it holds in $(\mathcal{C}, \Omega, \mathfrak{P})$ and since, for all $A,B \in S$, we have  $\homset_{S_{\downarrow f}}(A,B) := \homset_{\mathcal{C}}(A,B)$. To show Property \ref{property:SC2}, we must argue that that, for every diagram of the form
\begin{equation}\label{diagram:new_from_old_1}
\begin{tikzcd}
{Q \in f^{-1}(\Omega_n)} \arrow[r, "h_1"] \arrow[d, "h_2"] & H_1 \arrow[d] \arrow[r, "j_1"]                                                                                              & J_1 \arrow[dd]                                                       \\
H_2 \arrow[d, "j_2"] \arrow[r]                             & {\mathfrak{P}^S(h_1,h_2)} \arrow[rd, "{(j_1,j_2)}", dotted]                                                                 &                                                                      \\
J_2 \arrow[rr]                                             &                                                                                                                             & {\mathfrak{P}^S(j_1 h_1,j_2 h_2)}                                    \\
\end{tikzcd}
\end{equation}
in $S_{\downarrow f}$ there is an arrow $p$ (which is dotted in Diagram (\ref{diagram:new_from_old_1})) which makes the diagram commute.

By the second condition on $f$, we know that $f(\mathfrak{P}^S(j_1h_1,j_2h_2)) = \mathfrak{P}(fj_1h_1,fj_2h_2)$ and  $f(\mathfrak{P}^S(j_1h_1,j_2h_2)) = \mathfrak{P}(fj_1h_1,fj_2h_2)$. Thus we have that $f$ maps Diagram (\ref{diagram:new_from_old_1}) in $S_{\downarrow f}$ to the following diagram in $\mathcal{C}$.
\begin{equation}\label{diagram:new_from_old_2}
\begin{tikzcd}
\Omega_n \arrow[r, "h_1"] \arrow[d, "h_2"]                 & f(H_1) \arrow[d] \arrow[r, "j_1"]                                                                                           & f(J_1) \arrow[dd]                                                    \\
f(H_2) \arrow[d, "j_2"] \arrow[r]                          & {{f(\mathfrak{P}^S(h_1,h_2)) = \mathfrak{P}(fh_1, fh_2)}} \arrow[rd, "{f \circ (j_1,j_2) = (j_1,j_2)}" description, dashed] &                                                                      \\
f(J_2) \arrow[rr]                                          &                                                                                                                             & {{f(\mathfrak{P}^S(j_1h_1,j_2h_2)) = \mathfrak{P}(fj_1h_1,fj_2h_2)}}
\end{tikzcd}
\end{equation}
Since $(\mathcal{C}, \Omega, \mathfrak{P})$ satisfies Property \ref{property:SC2}, the dashed arrow $(j_1,j_2)$ in Diagram (\ref{diagram:new_from_old_2}) exists, is unique and makes the diagram commute. But since we have $\homset_{S_{\downarrow f}}(A,B) := \homset_{\mathcal{C}}(A,B)$ for all $A,B \in S$, we know that $p = (j_1,j_2)$, as desired. 

Now we will argue that $f$ is a spinal functor. By the first property of $f$, we know that $f$ preserves the spine. By the second property of $f$ and by what we just argued about Diagrams (\ref{diagram:new_from_old_1}) and (\ref{diagram:new_from_old_2}), we know that $f$ satisfies Property \ref{property:SF2} as well. Thus $f$ is a spinal functor from $(S_{\downarrow f}, \Omega^S, \mathfrak{P}^S)$ to $(\mathcal{C}, \Omega, \mathfrak{P})$.

Finally note that, since $f$ is a spinal functor from $(S_{\downarrow f}, \Omega^S, \mathfrak{P}^S)$ to $(\mathcal{C}, \Omega, \mathfrak{P})$, it must be that, if there exists an $S$-functor $G$ over $(\mathcal{C}, \Omega, \mathfrak{P})$, then the composition $G \circ f$ is an $S$-functor over $(S_{\downarrow f}, \Omega^S, \mathfrak{P}^S)$. Thus $(S_{\downarrow f}, \Omega^S, \mathfrak{P}^S)$ is measurable whenever $(\mathcal{C}, \Omega, \mathfrak{P})$ is.
\end{proof}
\end{theorem}

Theorem \ref{thm:new_spined_categories_from_old} allows us to easily define new spined categories from ones we already know. For example, denoting, for every graph $G$, the set of all functions of the form $f: V(G) \to [|V(G)|]$ as $\ell(G)$, consider the set $\mathcal{L} := \{\ell(G): G \in \mathcal{G}\}$ of all vertex-labelings of all finite simple graphs. Let $Q: \mathcal{L} \to \GRPH_{mono}$ be the surjection
\[Q : \; \bigl( f: G \to [|V(G)|] \bigr ) \longmapsto G/_{f}\]
which takes every labeling $f: G \to [|V(G)|]$ in $\mathcal{L}$ to the quotient graph \[G/_f := (V(G)/_f, E(G)/_f \setminus \{xx: x \in V(G)\}).\] Since $\GRPH_{mono}$ is a measurable spined category, by Theorem \ref{thm:new_spined_categories_from_old}, we know that $\mathcal{L}_{\downarrow Q}$ is also a measurable spined category. In particular, the triangulation number $\Delta_{\ell}$ of $\mathcal{L}_{\downarrow Q}$ takes any object $f: G \to [|V(G)|]$ in $\mathcal{L}_{\downarrow Q}$ to the tree-width of $G/_f$. 

This construction might seem peculiar, since it maps labeling functions (as opposed to graphs themselves) to tree-widths of quotiented graphs. Thus we define the \emph{$\ell$-tree-width of any graph $G$}, denoted $\tw_\ell(G)$, as $\tw_\ell(G) = \min_{f \in \ell(G)} \Delta_{\ell}[f]$. This becomes trivialy if we allow all possible vertex-labelings. However, by imposing restrictions on the permissible labelings, we can obtain more meaningful width-measures on graphs. We briefly consider two examples to demonstrate this principle.

\begin{example}[Modular tree-width]\label{example:modular_tree_width}
Recall that a vertex-subset $X$ of a graph $G$ is a called a \emph{module} in $G$ if, for all vertices $z \in V(G) \setminus X$, either $z$ is adjacent to every vertex in $X$ or $N(z) \cap X = \emptyset$. We call a labeling function $f: V(G) \to [|V(G)|]$ \emph{modular} if, for all $i \in [|V(G)|]$, the preimage $f^{-1}(i)$ of $i$ is a module in $G$. Thus, denoting by $\mathcal{M}$ the set $\mathcal{M} := \bigl \{\{\lambda:  \lambda \text{ is modular labeling of } G\}: G \in \mathcal{G} \bigr \}$ of all modular labelings, we obtain, as we did above, a spined category $\mathcal{M}_{\downarrow Q}$, where $Q$ is the function taking each modular labeling to its corresponding modular quotient. 

Note that the triangualation number of $\mathcal{M}_{\downarrow Q}$ maps every modular labeling to the tree-width of the corresponding modular quotient. Thus we can define \emph{modular tree-width} which takes any graph $G$ to the minimum tree-width possible over the set of all modular quotients of $G$.
\end{example}

\begin{example}[Chromatic tree-width]\label{example:chromatic_tree_width}
Denote the set of all proper colorings as $\mathbf{col} := \bigl \{\{\lambda:  \lambda \text{ is proper coloring of } G\}: G \in \mathcal{G} \bigr \}$. Then, as we just did in Example \ref{example:modular_tree_width}, we can study the spined category $\mathbf{col}_{\downarrow Q}$ and its triangulation number. Proceeding as before, this immediately yields the notion of \emph{chromatic tree-width}. 
\end{example}

\section{Further Questions}\label{sec:discussion}
As we have seen, spined categories provide a convenient categorial settings for the study of tree-like decompositions. 

Proxy pushouts occupy a middle ground between the \textit{amalgamation property} familiar from model theory (see e.g. Brody's dissertation~\cite{Brody2009PhD} for a thorough graph-theoretic treatment) and the amount of exactness available in e.g.  \textit{adhesive categories}~\cite{LackAdhesive}. The latter do not allow us to define width measures \textit{functorially} since they would rule out $\nat$ as a codomain for our functors (in particular poset categories are not adhesive). In contrast, $\nat$ has proxy pushouts and is a spined category.

Among spined categories, the measurable ones come equipped with a distinguished S-functor, the triangulation functor of Definition~\ref{def:triangulation-functor}, which can be seen as a general counterpart to the graph-theoretic notion of tree-width, and which gives rise to an associated notion of completion/decomposition. Moreover, Theorem~\ref{thm:triang-functor-is-S-functor} shows that the only possible obstructions to measurability are the \textit{generic} ones: if there is no obstruction so strong that it precludes the existence of \textit{every} S-functor, there can be no further obstruction preventing the existence of the triangulation functor.

Since most settings have only one obvious choice of structure-preserving morphism (which fixes the pushout construction as well), functoriality leaves the choice of an appropriate spine as the only ``degree of freedom''\footnote{How can we tell that we chose a good spine? Measurability provides a natural criterion!}. This makes spined categories an interesting alternative to other techniques for defining graph width measures, such as \textit{layouts}\footnote{Sometimes referred to as `branch decompositions' of symmetric submodular functions.} (used for defining branch-width \cite{robertsonX}, rank-width \cite{oum2006approximating}, $\mathbb{F}_4$-width \cite{kante2011}, bi-cut-rank-width \cite{kante2011} and min-width \cite{vatshelle2012new}), which rely on less easily generalized, graph-theory-specific notions of \textit{connectivity}. Finding algebraic examples of spined categories and associated width measures remains a promising avenue for further work. In particular, as we move from combinatorial structures towards algebraic and order-theoretic ones, choosing a spine becomes an abundant source of technical questions. 

\paragraph{Question.} Consider the category $\mathbf{Poset}_{oe}$ which has finite posets as objects and order embeddings as morphisms, equipped with the usual pushout construction. Is there a sequence of objects $n \mapsto \Omega_n$ which makes $\mathbf{Poset}_{oe}$ into a measurable spined category? \linebreak

Another promising topic for future work (which we describe in what follows) would be to `allow more complex spines'. Currently the spine of any spined category $\mathcal{C}$ consists of an $\mathbb{N}$-indexed sequence of objects $\Omega_1, \Omega_2, \dots$ which satisfies the requirement that for each object $X$ of $\mathcal{C}$, there exists at least one natural $n$ such that the homset $\homset_{C}(X, \Omega_n)$ is non-empty. Such sequences exist in many interesting categories involving combinatorial objects; however, in these settings there might not be an optimal (not to mention unique) such choice. For example, in the category having graphs as objects and topological minors as arrows, should we choose $(K_n)_{n \in \mathbb{N}}$ to be the spine or should we choose the sequence $(\mathbb{K}_n)_{n \in \mathbb{N}}$ whose $n$-th element is the graph obtained by subdividing a each edge of an $n$-clique $n$-times? A similar example (already observed by Wollan~\cite[Observation 1]{WOLLAN201547}) can be made for graph immmersions: for every graph $G$, there exist naturals $n$ and $\ell$ such that $G$ has an immersion to a star on $n$-leaves with $\ell$ parallel edges joining the center to each leaf. 
These considerations motivate our desire for an even further generalization of spined categories to a similar construct where the spine need not be an $\mathbb{N}$-indexed sequence; indeed this is a promising direction for future work. 

\bibliography{biblio}
\bibliographystyle{abbrv}

\end{document}